\documentclass[11pt,leqno]{article}

\usepackage{amssymb,amsmath,amsthm,mysects,url,rotating}
 \usepackage{graphicx,epsfig}
 \usepackage{xcolor,graphicx}
\newcommand{\n}{\noindent}

\newcommand{\vp}{\varepsilon}
\newcommand{\bb}[1]{\mathbb{#1}}
\newcommand{\cl}[1]{\mathcal{#1}}

\newcommand{\ovl}{\overline}

\theoremstyle{plain}
\newtheorem{thm}{Theorem}[section]
\newtheorem{lem}[thm]{Lemma}

\newtheorem{pro}[thm]{Proposition}

\newtheorem{cor}[thm]{Corollary}

\theoremstyle{definition}

\newtheorem{dfn}[thm]{Definition}

\theoremstyle{remark}
\newtheorem{rem}[thm]{Remark}

\numberwithin{equation}{section}

\def\tilde{\widetilde}

\renewcommand{\tilde}{\widetilde}

\def\U{{\cl U}}

\def\CC{\bb C}

\def\N{\bb N}

\def\B{\bb B}

\begin{document}
\def\d{\delta}

 \title{Strong convergence for reduced free products}

\author{by\\
Gilles  Pisier\\
Texas A\&M University and UPMC-Paris VI \footnote{Partially supported by  
    ANR-2011-BS01-008-01.}}

\maketitle
\begin{abstract}
Using   an inequality
 due to Ricard and Xu,
we give a different proof of Paul Skoufranis's  recent result
showing that the strong convergence of possibly non-commutative random variables $X^{(k)}\to X$  is stable
under reduced free product with a fixed non-commutative random variable $Y$. In fact we obtain a more general fact: assuming
that the families $X^{(k)}=\{X_i^{(k)}\}$ and $Y^{(k)}=\{Y_j^{(k)}\}$ are $*$-free as well as their   limits (in moments)  $X =\{X_i \}$ and $Y =\{Y_j\}$, the strong convergences
$X^{(k)}\to X$ and $Y^{(k)}\to Y$ imply that of 
 $\{X^{(k)},Y^{(k)} \}$   to $\{X ,Y\}$. Phrased in more striking language:
 the reduced free product is ``continuous" with respect to strong convergence. The analogue for 
 weak convergence (i.e. convergence of all moments) is obvious.
 Our approach extends to the amalgamated free product, left open by
 Skoufranis.
 \end{abstract}
 
 By a faithful $C^*$-probability space, we mean a unital $C^*$-algebra equipped with a state for which the GNS representation is faithful (it suffices for this that the state be faithful).
 We say that a family $\{X_m\mid m\in {\cl I} \}\subset A$
  generates $A$ if  $A$ is the smallest unital $C^*$-subalgebra
  of $A$ containing this family.  \\
  Let $(A^{(k)},\phi^{(k)})$ and $(A ,\phi)$ be faithful $C^*$-probability spaces ($k\ge 1$).
 Let $\{X^{(k)}_m\mid m\in {\cl I} \}\subset A^{(k)} $ (resp. $\{X_m\mid m\in {\cl I} \}\subset A$)   be    families
    generating $A^{(k)} $ (resp. $A$).
 \\ We say that  $\{X^{(k)}_m\mid m\in {\cl I} \}$
tends strongly  to  $\{X_m\mid m\in {\cl I} \}$ 
and we write
$$    \{X^{(k)}_m\mid m\in {\cl I} \} \stackrel{s}{\to} \{X_m\mid m\in {\cl I} \}   $$
if for any 
polynomial
   $P$ in the non-commutative variables $\{x_m,x_m^*\mid m\in {\cl I} \}$ 
   (these are called $*$-polynomials)
   we have  $\phi^{(k)}(P(X^{(k)}_m))\to \phi(P(X_m))$ (this is called the
   convergence in $*$-moments) and moreover
   $$\|P(X^{(k)}_m)\|\to \|P(X_m)\|.$$
   
   For Hermitian random matrices, when ${\cl I}$ is a singleton,
    this was called
 the phenomenon ``no eigenvalues outside (a small neighbourhood of) the support of
the limiting distribution" in \cite{BS}, where Bai and Silverstein obtained the case of  single  random 
 covariance (Hermitian)  $k\times k$-matrix; this was continued in \cite{PaS}.
See \S \ref{RM} below for a clarification 
of the meaning of strong convergence for
 a single Hermitian $k\times k$-matrix or more generally
  when the  families $\{X^{(k)}_m\mid m\in {\cl I} \}$ are formed of commuting Hermitian (or 
normal) operators.
 
 The notion of strong convergence, which was formally introduced in \cite{M},
was  inspired by Haagerup and Thorbj{\o}rnsen's paper \cite{HT}.
They prove there that if $\{ X_m^{(k)}(\omega)\mid m\in {\cl I} \}$ are independent   random $k \times k$-matrices the entries of which are all
 independent  complex Gaussian  $N(0,k^{-1})$, then for almost all $\omega$, the   random matrices 
$\{ X^{(k)}(\omega)\mid m\in {\cl I} \}$ tend strongly to a $*$-free
circular family
$\{X_m\mid m\in {\cl I} \}$.
More recent examples of strong convergence for infinite families
(i.e. when ${\cl I}$ is infinite) are obtained 
by Collins and Male in \cite{CM}: In particular, strong convergence
to the normalized Haar measure on the unit circle
also holds almost surely for i.i.d. families of unitary random matrices of size $k$ when $k\to \infty$. Additional examples   of strong convergence for families
can be found in Schultz's \cite{Sc} (real and symplectic Gaussian
random matrices, in other words   GOE and GSE),
and in Anderson's \cite{A} (Wigner matrices). Related error estimates appear 
 in \cite{HST}. 
 
 Strong convergence is also connected
 to operator space theory, via the so-called ``linearization tricks"
 from \cite{Pi} and \cite{HT}. We briefly describe this link in \S \ref{OS}. 
 
The results below are motivated by work by Camille Male \cite{M},
 who first considered the question of the stability of strong convergence, and by D. Shlyakhtenko's proof (see the appendix of \cite{M}) that the reduced free product with the $C^*$-algebra generated by free creators
 on the Fock space satisfies the desired stability property.
 Very recently, this was generalized by P. Skoufranis \cite{S}
 to essentially all reduced free products. In this note we give a different more direct proof based on an inequality due to \'E. Ricard and Q. Xu, which is a generalization to arbitrary reduced free products
 of results  proved previously by Voiculescu, Haagerup and Buchholz (see \cite{RX})
 for free products of groups. Our proof yields actually
 a stronger stability than the one appearing in \cite{S}, involving two limits as described in the abstract, but P. Skoufranis informed us that the original proof of \cite{S}
 also yields that improvement. In the final section, we extend our approach to the amalgamated free product, in answer to a question raised in \cite{S}.

 \section{A rough outline}
The main point to prove the result stated in the abstract is this:
if we are dealing with  
$P$ that is a polynomial in $X$'s and $Y$'s
that are $*$-free
and we want to compute its norm,
we observe that if $P$ is of (joint) degree at most $d$
then $Q=(P^*P)^m$ will be of degree at most $2md$.

The Ricard-Xu non-commutative Khintchine inequality (\cite{RX})
is
 \begin{equation}\label{x11}(4d)^{-1} kh(P) \le \|P\| \le (2d+1)^2 \ kh(P).\end{equation}

This gives us
$$(8md)^{-1} kh(Q) \le \|Q\| \le (4md+1)^2 \ kh(Q)$$
where $kh(Q)$ is a certain expression
(actually  a norm depending on $md$) that we will need to analyse below.
\\
Fix $\vp>0$.
The last inequality gives us
that if $m=m(d,\vp)$ is fixed but chosen large enough so that $(\max\{8md,(4md+1)^2\})^{1/2m}<1+\vp$ 
then we have
$$(1+\vp)^{-1} [kh((P^*P)^m)]^{1/2m}\le \|P\|\le (1+\vp)[kh((P^*P)^m)]^{1/2m}.$$

Thus to show the strong convergence
of ${X^{(k)}},Y^{(k)}$ to ${X },Y$ it suffices to show
that for $m$ fixed and $Q=(P^*P)^m$ we have
$$[kh(Q({X^{(k)}},Y^{(k)})]^{1/2m} \to  [kh(Q({X^{}},Y)]^{1/2m}   ,$$
or merely 
 \begin{equation}\label{00}kh(Q({X^{(k)}},Y^{(k)})) \to  kh(Q({X},Y))   .\end{equation}
But now a closer look at $ kh(Q({X},Y))$  in \S \ref{sec3} will show that
 this holds.
 
  \section{GNS construction and a specific notation}
 
 We choose the convention to have all inner products $\langle y,x\rangle$
  linear $x$ and antilinear in $y$.\\
In the sequel,   we denote by $X\otimes Y$ the algebraic tensor product
of two Banach spaces.  \\
Given Hilbert spaces $H,\cl H$ and 
 $C^*$-subalgebras $A\subset B(H)$
and  $B\subset B({\cl H})$ we denote as usual by $A\otimes_{\min} B$
the closure of $A\otimes B$ in the space $B(H\otimes_2 {\cl H})$, and by
$\|\cdot\|_{\min} $ the induced norm.

 \subsection{}\label{b} Let ${\cl A}$ be a unital $*$-algebra,  assumed sitting inside some ambient unital $C^*$-algebra. In the sequel, we  always   make this assumption
for our unital $*$-algebras.
  By  a state on ${\cl A}$ we mean
a linear functional such that $\phi(1)=1$ and $\phi(x^*x)\ge 0$ for all $x\in {\cl A}$.
Given this, the classical GNS construction produces a Hilbert space denoted by $L_2(\phi)$
and a $*$-homomorphism $\pi_\phi:\ {\cl A}\to B(L_2(\phi))$ equipped with a distinguished cyclic unit vector
 $ \xi_\phi\in L_2(\phi)$, such that
$\phi(c)=\langle  \xi_\phi,\pi_\phi(c) \xi_\phi\rangle$ for any $c\in {\cl A}$.
Let  $\pi=\pi_\phi$ and $\xi=\xi_\phi$ for simplicity. Let $A=\overline{\pi({\cl A})}\subset B(L_2(\phi))$.
Then $A$ is a unital $C^*$-algebra. Let 
 $\hat \phi(a) =\langle \xi,\pi(a) \xi\rangle$ for any $a\in A$. Then $\hat \phi$ is a state 
 on $A$, $L_2(\hat\phi)\simeq L_2(\phi)$  and the representation $A\subset B(L_2(\phi))$ can be identified with
 the result of the GNS construction applied to $(A,\hat \phi)$. We view the original algebra 
 ${\cl A}$ as acting on $L_2(\phi)$ by the correspondence $a\mapsto \pi(a)$. Although this action
 may be non injective on  ${\cl A}$,
  the resulting GNS representation $A\subset B(L_2(\phi))$  is (by definition)  faithful  on $A$.
\subsection{}\label{b2}
  
  Let $H=L_2(\phi)$.  For any $x\in B(H)$ we denote by ${}^{t}x \in B(H^*)$
  the adjoint operator. 
  Let $A^{op}$ denote the opposite of $A$ i.e. the same 
  as $A$ but with reverse multiplication (i.e. we set $a\cdot b=ba$).
  We then 
  define $\pi^{op}: \ A^{op}\to B(H^*)$ by
  $\pi^{op}(a)=  {}^{t}\pi(a) \in B(H^*)$. Note that  $\pi^{op}$
  is a $*$-homomorphism on $A^{op}$. Let
  $  \xi^{op}\in H^*$ denote the linear form on $H$ defined by
  $$\forall h\in H\quad  \xi^{op}(h)=\langle \xi, h\rangle.$$
  Note that $  \xi^{op}\in H^*$ could be identified with $\bar \xi\in \bar H$.
  Moreover,  $\pi^{op}$ can be viewed as the GNS representation
  associated to $\phi$   viewed as a state on $A^{op}$, with cyclic vector
  $\xi^{op}$. 
  
  \n{\bf Notation.} In the sequel we will work with
  a dense $*$-subalgebra ${\cl A}\subset A$ .
  It will be convenient to use the following notation valid 
  for all $x\in A$, but used mostly  for all $x $ in ${\cl A}$:
       \begin{equation}\label{x22}  \pi(x)\xi =x \xi   \ {\rm and} \ \pi^{op}(x)\xi^{op}=\xi^{op} x.\end{equation}
   \subsection{}\label{b3-} (A notation for further reference)
  Fix an integer $d\ge 1$. Consider   a 
   linear subspace $X\subset A$. 
  Let  $X^{\otimes d} $ denote the algebraic tensor product. Let $0\le r\le  d$.
   We define a linear mapping
  $$t_r: \ X^{\otimes d} \to H \otimes H^*$$
  by
  $$\forall x=a_1\otimes \cdots\otimes a_d\in  X^{\otimes d}\quad
  t_r(x)=a_1  \cdots  a_r \xi \otimes  \xi^{op} a_{r+1}  \cdots  a_d .$$
In the extreme cases $r=0$ or $r=d$, we mean
$$ 
  t_0(x)=   \xi \otimes \xi^{op} a_{1}  \cdots  a_d 
  \quad{and}\quad
  t_d(x)= a_1  \cdots  a_d \xi \otimes  \xi^{op}.$$
  This definition is extended to the whole of $X^{\otimes d} $ by linearity.
  \subsection{}\label{b3} (More notation)
  Assume now that   the  linear subspace $X\subset A$  
  is included in the direct sum of two subspaces $X_1,X_2\subset A$,
  so that $X\subset X_1+X_2$ and we have a linear embedding
  $J:\ X\to X_1\oplus X_2$.\\
  For further reference, we define for any  $1\le r\le  d$ a linear mapping
  $$s_r: \ X^{\otimes d} \to H \otimes H^*\otimes (X_1\oplus X_2)$$
  by setting
  $$\forall x=a_1\otimes \cdots\otimes a_d\in  X^{\otimes d}\quad
  s_r(x)=  a_1  \cdots  a_{r-1}  \xi \otimes \xi^{op} a_{r+1}  \cdots  a_d \otimes J(a_r).$$
  In the extreme cases $r=1$ or $r=d$, we mean
$$ 
  s_1(x)=   \xi \otimes  \xi^{op} a_{2}  \cdots  a_d  \otimes J(a_1)
  \quad{and}\quad
  s_d(x)=  a_1  \cdots  a_{d-1}  \xi \otimes  \xi^{op} \otimes J(a_d).$$
  This definition is extended to the whole of $X^{\otimes d} $ by linearity.

  \section{Background on  ultraproducts}

It will be convenient to use   ultraproducts,
but we only need the very basic 
and elementary facts that are recalled below.

 \subsection{}\label{ul}
Let $\U$ be a non trivial ultrafilter on $\N$.
Given a sequence $(X^{(k)})$ of Banach spaces,
their ultraproduct is usually denoted by
$$\prod_{k\in \N} X^{(k)}/{\U}.$$
We will more often denote it by $X^{\U}$   (see
e.g. \cite{He} for more background information).
The elements $x\in X^{\U}$ are equivalence classes
of bounded sequences $(x^{(k)})$ with $x^{(k)}\in X^{(k)}$ for all $k$.
By definition, two such sequences $(x^{(k)})$  , $(y^{(k)})$ are equivalent
if $\lim\nolimits_\U\|x^{(k)}-y^{(k)}\|=0$.
We will sometimes write $x=[x^{(k)}]_\U$ to denote that  $(x^{(k)})$
is a representative of $x$. Whenever this holds we have
$\|x\|_{X^\U}=\lim\nolimits_\U \|x^{(k)}\|$.
\subsection{}\label{wo2} Let $(X^{(k)}),(Y^{(k)})$ be sequences of Banach spaces (resp. unital $C^*$-algebras). Let $T  ^{(k)}:\ X^{(k)}\to Y^{(k)}$ be a bounded sequence of linear mappings (resp. unital $*$-homomorphisms)
then the mapping $T^\U:\ X^{\U}\to Y^{\U}$ defined whenever
$x=[x^{(k)}]_\U$ by $T^\U(x)=[T  ^{(k)}(x^{(k)}) ]_\U$ is
bounded (resp. a  unital $*$-homomorphism) with
$\|T^\U\|=\lim\nolimits_\U \|T  ^{(k)}\|$.

\subsection{}\label{wo}
As is well known, when all the spaces in $(X^{(k)})$ are Hilbert spaces
$X^{\U}$ is also a Hilbert space.\\
It may be worthwhile to remind the reader
that if $(Y^{(k)})$ is another family of  Hilbert spaces, we have a canonical isometric
embedding  
$$X^{\U}\otimes_2 Y^{\U}\subset \prod_{k\in \N}  X^{(k)}\otimes_2 Y^{(k)} /\U,$$ 
taking $ [x^{(k)}]_\U \otimes [y^{(k)}]_\U  $ to $[x^{(k)}\otimes y^{(k)}]_\U  $.
Of course this extends to an arbitrary finite number of factors.
Moreover, if $\sup_k \dim(X^{(k)})<\infty$, then this embedding is an isomorphism.

  \subsection{}\label{ulbis} Let $(H^{(k)})$ be a sequence of Hilbert spaces.
 We have then an isometric identification
 $${{H^\U}^{*}}=\prod\nolimits_{k\in \N} {H^{(k)}}^*/{\U}.$$
Let ${H^{(k)}}\otimes  {H^{(k)}}^*\subset B(H^{(k)})$ be the usual embedding
(taking $x\otimes f$ to the mapping $h\mapsto xf(h)$).
  Then,     
  for any sequence $t^{(k)}\in  {H^{(k)}}\otimes  {H^{(k)}}^*$
  with  ranks uniformly bounded by some number $N$, 
  associated, as in \ref{wo}, to some $t \in {H^\U}\otimes {H^\U}^*$ (of rank at most $N$) 
 we have
   \begin{equation}\label{x12}\lim\nolimits_{\U} \| t^{(k)}  \|_{B({H^{(k)}})}= \| t\|_{B(H^\U)}.\end{equation}
   More explicitly,
   let $t=\sum_1^N h(\alpha)\otimes \xi(\alpha)\in {H^\U}\otimes {H^\U}^*$.
 Assume $ h(\alpha)=[h^{(k)}(\alpha)]_\U$ and
 $ \xi(\alpha)=[\xi^{(k)}(\alpha)]_\U$. Let
 $t^{(k)}=\sum\nolimits_1^N h^{(k)}(\alpha)\otimes \xi^{(k)}(\alpha)\in {H^{(k)}}\otimes {H^{(k)}}^*$. Then \eqref{x12} holds.

 For the convenience of the reader, let us sketch  a quick
 and instructive verification
  of \eqref{x12}.
Let $E\subset {H^\U}$ and $F\subset {H^\U}^*$ be $N$-dimensional subspaces such that $t\in E\otimes F$. Let 
$(e_i)=(e^{(k)}_i)$ and $(f_j)=(f^{(k)}_j)$ be orthonormal bases  of $E$ and $F$.
We can then  write
  $t=\sum a_{ij} e_i\otimes f_j $
  and  $t^{(k)}=\sum a_{ij} e^{(k)}_i\otimes f^{(k)}_j $.
  By an elementary perturbation, we may assume that
  $(e^{(k)}_i)$ and $(f^{(k)}_j)$ are orthonormal
  in  ${H^{(k)}}$ and ${H^{(k)}}^*$ for all $k$ large enough.
  Then $\| t^{(k)}  \|_{B({H^{(k)}})}= \| t\|_{B(H^\U)}=\|[a_{ij}]\|_{M_N}$
  for all $k$ large enough.
  
\subsection{}\label{u} 
Let $({\cl H}^{(k)})$ be a sequence of Hilbert spaces.
Let $S^{(k)}\in B({\cl H}^{(k)})$ be a bounded sequence.
Let $S^{\U}\in B({\cl H}^{\U})$ be the associated operator.
By \ref{wo2}, we already know that
$\|S^{\U}\|=\lim\nolimits_\U \| S^{(k)} \|.$\\
More generally, let $N$ be a fixed integer and $H$ a Hilbert space. We denote by $M_N(B(H))$
the space of $N\times N$ matrices with entries in $B(H)$ with the usual  norm.\\
Let $[a_{ij}^{(k)}]\in M_N(B({\cl H}^{(k)}))$ be a bounded sequence.
Note that if $K^{(k)}={\cl H}^{(k)}\oplus\cdots\oplus {\cl H}^{(k)}$ ($N$-times),
there is a natural identification $K^{\U}={\cl H}^{\U}\oplus\cdots\oplus {\cl H}^{\U}$ ($N$-times).
 Then clearly
\begin{equation}\label{ida1}\|[a_{ij}^{\U}]\|_{M_N(B({\cl H}^{\U}))  }=\lim\nolimits_\U \|[a_{ij}^{(k)}]  \|_{M_N(B({\cl H}^{(k)}))}.\end{equation}
Consider now 
$C^*$-subalgebras $B^{(k)}\subset B({\cl H}^{(k)})$, and their
ultraproducts $B^\U\subset B({\cl H}^\U)$. Let
  $$s=\sum\nolimits_1^N h(\alpha)\otimes \xi(\alpha)\otimes \beta(\alpha) \in {H^\U}\otimes {H^\U}^*\otimes  B^\U.$$
  Assume $ h(\alpha)=[h^{(k)}(\alpha)]_\U$, 
 $ \xi(\alpha)=[\xi^{(k)}(\alpha)]_\U$
 and    $ \beta(\alpha)=[\beta^{(k)}(\alpha)]_\U$.
  Then let
  $$s^{(k)}=\sum\nolimits_1^N h^{(k)}(\alpha)\otimes \xi^{(k)}(\alpha)\otimes \beta^{(k)}(\alpha)
  \in  {H^{(k)}}\otimes  {H^{(k)}}^* \otimes B^{(k)}.$$ 
We claim that
\begin{equation}\label{ida2}\|s\|_{B({ H}^{\U})  \otimes_{\min} B^{\U}}= 
\lim\nolimits_{\cl U}    \| s^{(k)}  \|_{B({ H}^{(k)})  \otimes_{\min} B^{(k)} }.\end{equation}
Arguing as in
\ref{ulbis} we can find orthonormal systems
$(e^{(k)}_i)$ and $(f^{(k)}_j)$ of length $N$
in  ${H^{(k)}}$ and ${H^{(k)}}^*$
with respect to which we may write
$s^{(k)}=\sum_{ij} e^{(k)}_i\otimes f^{(k)}_j \otimes a^{(k)}_{ij} $.
We have then
$$ \| s^{(k)}  \|_{B({ H}^{(k)})  \otimes_{\min} B^{(k)} }=\|s^{(k)}\|_{B({ H}^{(k)})  \otimes_{\min} B({\cl H}^{(k)}) }
=\| [a^{(k)}_{ij}]  \|_{M_N(B({\cl H}^{(k)}))} ,$$
and similarly
$\|s\|_{B({ H}^{\U})  \otimes_{\min} B^{\U}}=\|s \|_{B({ H}^{\U})  \otimes_{\min} B({\cl H}^{\U}) }
=\| [a^{\U}_{ij}]  \|_{M_N(B({\cl H}^{\U}))} $. Now the claim follows
from \eqref{ida1}.
\subsection{}\label{e}
Given a sequence of states $\phi^{(k)}$
on a unital $*$-algebra ${\cl A}$, let $(\pi^{(k)},H^{(k)},\xi^{(k)})$ be the associated
GNS construction and let $A^{(k)}=\overline{\pi^{(k)}({\cl A})}\subset B(H^{(k)})$ be the associated 
$C^*$-algebra.
 Let $\pi^\U: \ {\cl A} \to B(H^\U)$
be the representation  defined for any $z=[z^{(k)}]_\U\in H^\U$ and $b\in {\cl A}$ by 
$$\pi^\U(b) ([z^{(k)}]_\U)= [\pi ^{(k)}(b)(z^{(k)})]_\U.$$
Obviously,
$$\|\pi^\U(b) \|=\lim\nolimits_\U\|\pi ^{(k)}(b)  \|.$$
Let   $\phi  =\lim\nolimits_\U \phi^{(k)}$ 
relative to    \emph{pointwise} convergence on ${\cl A}$, 
let $\pi:\ {\cl A}\to B(L_2(\phi))$ be the associated GNS representation and let 
$\xi=\xi_\phi$. Let also $\xi^\U=[\xi^{(k)}]_\U$. Then
   $$\|\pi^\U(b) \xi^\U\|^2=\lim\nolimits_\U \|\pi^{(k)}(b) \xi^{(k)}\|^2=\lim\nolimits_\U \phi^{(k)}(b^*b) =\phi(b^*b)=
\|\pi(b)\xi\|^2.$$
Similarly, using the identity ${{H^\U}^{*}}=\prod_{k\in \N} {H^{(k)}}^*/{\U}$ (see \ref{ulbis}),
   we have for any $b\in {\cl A}$
    $$\|{\pi}^{\U\, op}(b) {\xi}^{\U \,op}\|_{{H^\U}^{*}}^2=\lim\nolimits_\U \|{\pi}^{{(k)}\, op}(b) {\xi}^{{(k)} \, op}\|_{{H^{(k)}}^{*}}^2=\lim\nolimits_\U \phi^{(k)}(bb^*) =\phi(bb^*)=
\|{\pi}^{op}(b){\xi}^{op}\|^2.$$

It is natural to extend the notation \eqref{x22} by setting
\begin{equation}\label{x23}\pi^{(k)}(b) \xi^{(k)}= b \xi^{(k)} \text{ and }  {\pi}^{{(k)}\, op}(b) {\xi}^{{(k)}\, op} =
  {\xi}^{(k)} b,\end{equation}
\begin{equation}\label{x24}\pi^{\U}(b) \xi^{\U}= b \xi^{{\U}} \text{ and }  {\pi}^{{{\U}}\, op}(b) {\xi}^{{{\U}}\, op} =
  {\xi}^{{\U}} b,\end{equation}

\subsection{}\label{x1} Therefore, the correspondence
$b\xi\mapsto b \xi^\U$ extends to an isometric isomorphism
from $L_2(\phi)$ onto the subspace $K^\U\subset H^\U$ that is the closure
of $\{b \xi^\U \mid b \in {\cl A}\}$. More precisely, the restriction of $\pi^\U$ to $K^\U$, i.e. $b\mapsto 
\pi^\U(b)_{|K^\U}\in B(K^\U)$,
  is unitarily equivalent to  the representation $\pi=\pi_\phi $.\\ 
  Similarly, $ {\xi}^{op} b  \mapsto   {\xi}^{\U \,op} b $ extends to an isometric isomorphism
from $L_2(\phi)^*$ onto a subspace of ${{H^\U}^{*}}$,
which can be identified isometrically, via $y\mapsto y_{|K^\U}$ with ${{K^\U}^{*}}$.

  \subsection{}\label{x10} Now let us assume moreover  that  $\phi  =\lim\nolimits_\U \phi^{(k)}$ \emph{strongly}. This means  
(see below)
that $\|\pi(b)\|=\lim\nolimits_\U \|\pi^{(k)}(b)\|=\|\pi^\U(b) \|$ for any $b\in {\cl A}$.
Then the mapping $\pi(b)\mapsto \pi^\U(b)$
defines an isometric (and automatically completely isometric) 
embedding of $C^*$-algebras
$$\psi:\ A=\overline{\pi({\cl A})}\to B(H^\U).$$

   \subsection{}\label{rem1} 
   Let $(H_i,\xi_i)_{i\in I}$ be a family of Hilbert spaces, each equipped with a distinguished unit vector. 
    Let $(H,\xi)=\ast_{i\in I}(H _i,\xi _i)$ be their free product
    in the sense of \cite{VDN}. This is defined
    as
    $$(H,\xi)=(H_0 \oplus \oplus_{d\ge 1} H_d, \xi)$$
    where $H_0=\CC$ with unit vector $\xi=1_{\CC}$ (viewed as sitting in $H$) and
    $$H_d=\oplus_{i(1)\not=\cdots\not=i(d)} [H_{i(1)}\ominus \CC \xi_{i(1)}]\otimes_2\cdots\otimes_2  [H_{i(d)}\ominus \CC \xi_{i(d)}] .$$
    It is natural to wonder whether this free product commutes with utraproducts. Let $(H^{(k)}_i,\xi^{(k)}_i)_{i\in I}$ be a sequence of such families
   (indexed by   $k\in \N$).
   Let
    $(H^{(k)},\xi^{(k)})=\ast_{i\in I}(H^{(k)}_i,\xi^{(k)}_i)$ . Going back to the definition of the free product, a moment of thought (recall \ref{wo})
 shows that  we have a canonical isometric 
 embedding
 \begin{equation}\label{10}  {\chi}: *_{i\in I}  {  H}^{\U}_i \subset  { H}^{\U} \end{equation}
 that respects the distinguished vectors.\\
Assuming $I=\{1,2\}$, the mapping ${\chi}$ can be described like this:
First we have ${\chi}(\xi)=  \xi^{\U}$, then
whenever we consider an element 
$x_j$    in  ${ H}^{\U}_1 \cap \{ \xi^{\U}_1 \}^\perp $ 
(resp. ${ H}^{\U}_2\cap \{ \xi^{\U}_2 \}^\perp $) we can choose 
representatives $(x^{(k)}_j)$ of $x_j$ with $x^{(k)}_j$    in  ${ H}^{(k)}_1\cap \{ \xi^{(k)}_1 \}^\perp $ (resp. ${ H}^{(k)}_2\cap \{ \xi^{(k)}_2 \}^\perp $), so that
given an element $x=x_1\otimes\cdots\otimes x_d$
of degree $d$ in  $*_{i\in I}  { H}^{\U}_i$, with  alternating factors in  ${ H}^{(k)}_1\cap \{ \xi^{(k)}_1 \}^\perp $ and ${ H}^{(k)}_2\cap \{ \xi^{(k)}_2 \}^\perp $,   we then define ${\chi}(x)$ as the 
element of   ${ H}^{\U} $ admitting as representative the sequence $(x^{(k)})$ with
$x^{(k)} =x^{(k)}_1\otimes\cdots\otimes x^{(k)}_d$. However, it is easy to
see that this embedding $\chi$ is not surjective.
   
 \section{Main result}\label{sec3}
  We now turn to a more formal description of our main result.\\
  A more abstract (but equivalent) version of the  statement
   in the abstract
can be given in terms of convergence of states. We use the notation in \ref{b}.

\begin{dfn} Let $\phi$ (resp. $\phi^{(k)}$,  $(k\in \N)$) be   states on a unital $*$-algebra ${\cl A}$ (assumed included in some $C^*$-algebra)
with 
 associated GNS Hilbert spaces denoted by $L_2(\phi)$ (resp.  $L_2(\phi^{(k)})$).
Let $\pi $ (resp. $\pi^{(k)}$) be the associated GNS representations of ${\cl A}$ on
these Hilbert spaces.
We say that  $\phi^{(k)}$ tends to  $\phi$ strongly and we write
$\phi^{(k)} \stackrel{s}{\to} \phi$ if $\phi^{(k)}$ tends to  $\phi $ pointwise
on ${\cl A}$ and moreover if
$\|\pi^{(k)}(c)\|\to \|\pi(c)\|$ for any $c$ in ${\cl A}$. 
\end{dfn}

In \cite{VDN} the notion of free product of a family of states 
is defined. It can be described as follows. Consider a family of states
$\{\phi_i\mid i\in I\}$ with  
GNS Hilbert space $H_i=L_2(\phi_i)$, GNS representation
$\pi_i:\ {\cl A}_i \to B(L_2(\phi_i)   )$
and distinguished unit vector $\xi_i$. Let $A_i\subset B(L_2(\phi_i))$ be the associated $C^*$-algebra.
We denote by  $\hat \pi_i:\ A_i\to B(L_2(\phi_i)   )$ the inclusion map. 
Let ${\cl A}=*_{i\in I} {\cl A}_i$ be the (algebraic)
free product of unital $*$-algebras. Following Voiculescu (see \cite{VDN})
one defines a Hilbert space free product $(H,\xi)= *_{i\in I}( H_i,\xi_i)$
and a representation $\pi$ of ${\cl A}$
acting on $(H,\xi)$. Let
$\hat \phi_i$ (resp. $\hat \phi$) be the vector state on $A_i$
(resp. $\overline{\pi({\cl A})}$) associated to $\xi_i$ (resp. $\xi$).
The unital $C^*$-subalgebra  $A=\overline{\pi({\cl A})}\subset B(H)$, equipped with
$\hat \phi$, 
 is called the  reduced free product of  $(A_i, \hat \phi_i )_{i\in I}.$
 Note that, by \cite{Dy},   $\hat \phi$ is faithful on $A$
 if each $\hat \phi_i$ is  faithful on $A_i$.
 
We will denote
  by $\phi=*_{i\in I} \phi_i$ the vector state on ${\cl A}$ defined
  by $\phi(b)=\langle \xi, \pi(b)\xi\rangle$.  We call it
  the free product of the states $\{\phi_i\mid i\in I\}$.
  Then we can reformulate the main result like this:
  
 \begin{thm}\label{t3} Let ${\cl A}_i$ ($i\in I$) be a family of unital $*$-algebras.
 Let $\{\phi^{(k)}_i\mid i\in I\}$  ($k\in \N$) be a sequence of families of states,
 each $\phi^{(k)}_i$ being a state
 on ${\cl A}_i$.  
 Assume that  
 we have states $\phi_i$ on ${\cl A}_i$ such that, 
 for each $i\in I$, when $k\to \infty$ we have  
 $$ \phi^{(k)}_i \stackrel{s}\to \phi_i.$$
 Then $$*_{i\in I}\phi^{(k)}_i \stackrel{s}\to *_{i\in I}\phi_i .$$
 
 \end{thm}
  \subsection{}\label{le0-1}  
  The analogue of the preceding statement for pointwise convergence
  of states is obvious from the definition of  the  reduced free product
  in \cite{VDN}. 

\subsection{}\label{le0} Let $A_i$ be associated to $({\cl A}_i,\phi_i)$ by the GNS
construction as above.
We view each $A_i$ as a subalgebra of  the  reduced free product $A= {*}_{i\in I} A_i$.
Let $$\stackrel{\circ}{A_i}=\{x\in A_i\mid \phi_i(x)=0\}.$$
By a monomial of degree $d$
 we mean a   product of the form
 $x_1\cdots x_d$
 with $x_j\in \stackrel{\circ}{A_{i_j}}$   such that $i_1\not=i_2\not= \cdots\not= i_d$.
By a homogeneous
element of degree $d$ in $A= {*}_{i\in I} A_i$ we mean a finite sum of monomials
 of degree $d$. 
 An element is called of degree $\le d$ if it is
 a sum of homogeneous
elements each of degree $\le d$.
 Note that the elements of finite degree are dense in  $A$. 
 
Let us denote by $W_{d}$ (resp. $W_{\le d}$) the space of homogeneous elements of degree $d$ (resp. ${\le d}$) in the preceding sense. We also set $W_0=\CC 1$ and denote by $\overline{W_{d}}$
the closure of $W_{d}$  in $A$.  
  Then the Ricard-Xu inequality we will use  is this
  (note that, by our convention in \ref{b}, the assumption in \cite{RX} that all the GNS constructions are  faithful
  is here automatic):   
There are constants $c'>0$ and $\beta>0$
such that \begin{equation}\label{rx0}
\forall d\ \forall x\in W_{d}  \quad  (c' d^\beta)^{-1}kh(x)\le  \|x\| \le c' d^\beta kh(x),\end
 {equation}
 where we set
\begin{equation}\label{rx0+} kh(x)=  \max \{ \max_{0\le r\le d} \|t_r(x)\|,  \max_{1\le r\le d}  \|s_r(x)\|  \},\end
 {equation} and where $t_r(x),s_r(x)$ are defined as follows:
We assume $I=\{1,2\}$ for notational simplicity.
Let $X=   \stackrel{\circ}{A_1}+\stackrel{\circ}{A_2} \subset A $. We have obviously an embedding denoted by $x\mapsto [x]$ of
 $W_{d}$ into $X^{\otimes d} $. So following \ref{b2} we may set
 for any $x\in W_{d}$
 $$t_r(x)=t_r([x]).$$
We will  identify an element $\sum x_j\otimes y_j\in  H\otimes H^*$ with the linear
 map $T\in B(H)$ defined by $T(z)=\sum x_j \otimes y_j(z)$. In this way
 we will view
 $t_r(x)$ as an element of $B(H )$, and we denote by $\|t_r(x)\|$ its norm.
   
  Let  $X = \stackrel{\circ}{A_1}+\stackrel{\circ}{A_2}$, $X_i=A_i$   and 
 $J:\     \stackrel{\circ}{A_1}+\stackrel{\circ}{A_2} \to  {A_1}\oplus {A_2}$  
   be the canonical  embedding. 
   Following \ref{b3}, we set\\
  $$s_r(x)= s_r([x])\in H \otimes H^*\otimes ({A_1}\oplus{A_2}).$$
  We then denote by  $\|s_r(x)\|$ its norm in the minimal tensor
  product $ B( H) \otimes_{\min} [  A_1  \oplus  A_2    ]$.
  Equivalently, if we are given
  isometric representations $\psi_j:\ A_j \to B(H_j)$ ($j=1,2$) this is the maximum of two norms,
  one in  $ B( H) \otimes_{\min}   A_1$ 
  (induced by  $B(H \otimes H_1  )$) and one in 
  $ B(H  ) \otimes_{\min}   A_2$  (induced by  $B(H \otimes H_2)$).

 \subsection{(On recentering)}\label{cor}  For any $i\in I$, let $  \stackrel{\circ}{{\cl A}_i}=\{ x\in {\cl A}_i\mid \phi_i(x)=0\}$.
The elements of $  \stackrel{\circ}{{\cl A}_i}$ are sometimes called ``centered" (with respect to $\phi_i$).
 Let ${\cl A}_0=\CC 1$. 
 By elementary (free) algebra, one can show that
 the algebraic free product ${\cl A}$ is linearly  isomorphic to the direct sum
  \begin{equation}\label{72}{\cl A}_0\oplus \oplus_{d\ge1,  i_1\not=\cdots\not=i_d} {\cl A}(i_1,\cdots,i_d)\end{equation}
where the 
 subspaces ${\cl A}(i_1,\cdots,i_d)$ ($d\ge1$, $i_i\not=i_2\not=\cdots$) are formed
 of all products
 of the form
 \begin{equation}\label{71}x_{i_1} \cdots x_{i_d}
   \quad {\rm with}  \quad  x_{i_j} \in 
 \stackrel{\circ}{{\cl A}_{i_j}}\quad    \forall 1\le j\le d . 
 \end{equation}

 Recall that, for each $k$, we are given a state $\phi_i^{(k)}$ on ${\cl A}_i$. We will
 denote by $v_i^{(k)}: {\cl A}_i \to {\cl A}_i$ the linear mapping that transforms
 centering with respect to $\phi$ into centering with respect
 to $\phi_i^{(k)}$. More precisely, $v_i^{(k)}$ is 
 defined by  $v_i^{(k)}(1)=1$ and 
 $\forall a\in  \stackrel{\circ}{{\cl A}_{i}}\ v_i^{(k)}(a)= a - \phi_i^{(k)}(a)1$ . 
 Using the above  direct sum decomposition of ${\cl A}$, 
 and viewing ${\cl A}_i \subset {\cl A}$ we can extend the mappings 
 $v_i^{(k)}$ to a single mapping on ${\cl A}$. More precisely,
 using the freeness of the product, there is a unique linear map
 $v^{(k)}: {\cl A} \to {\cl A}$  that coincides with $v_i^{(k)}$ on  ${\cl A}_i$
 for each $i\in I$ and is such that for any element of the form
 \eqref{71} we have
  \begin{equation}\label{x3}v^{(k)} (x_{i_1} \cdots x_{i_d})=v^{(k)}_{i_1}(x_{i_1})  \cdots v^{(k)}_{i_d}(x_{i_d})
 . \end{equation}
 We note that 
 if we equip ${\cl A}$ with the maximal $C^*$-norm then,
 since $\phi_i^{(k)}\to\phi_i$, we clearly have
  \begin{equation}\label{as55}  \forall b\in {\cl A}\quad \| v^{(k)}(b)- b\|\to 0.\end{equation}
 Let us denote by  ${\cl P}_d$  the linear projection (relative to \eqref{72})
 from ${\cl A}$ to the subspace $\cl W_d\subset {\cl A} $
 defined by 
 $$\cl W_d=\oplus_{ i_1\not=\cdots\not=i_d} {\cl A}(i_1,\cdots,i_d) .$$
 Fix $k$. Again let $\cl W_0=\CC1$ and $\cl W_{\le d}=\cl W_0+\cdots +\cl W_d $.\\
  Suppose now that we replace $\phi$ by $\phi_i^{(k)}$ so that 
 we have a direct sum decomposition as above but now associated to $\phi_i^{(k)}$.
 This leads to subspaces $\cl W^{(k)}_d\subset {\cl A} $
 defined exactly like $\cl W_d$ but with respect to $\phi_i^{(k)}$. Let
  ${\cl P}^{(k)}_d$ denote   the linear projection
 from ${\cl A}$ to the subspace $\cl W^{(k)}_d\subset {\cl A} $ in the said direct sum decomposition
 relative to $\phi_i^{(k)}$.  It is easy to check that we have for any $k$
   \begin{equation}\label{as56} 
    v^{(k)} {\cl P}_d= {\cl P}^{(k)}_d v^{(k)}
 .\end{equation}

     \subsection{}\label{le00} By \cite[Cor. 3.3]{RX} ${\cl P}_d$ extends by density to  a  completely  bounded projection
     from $ {*}_{i\in I} A_i $ to $ \overline{W_d}$, that we will still denote
     abusively by $\cl P_d $ satisfying:
 $$ \|{\cl P}_d\|_{cb}\le \max\{1,4d\}.$$
Obviously  this implies  
 \begin{equation}\label{as4} 
 (\max\{1,4d\})^{-1} \max_{0\le d\le D}\{\|{\cl P}_d(x)\|\}\le \|x\|\le (d+1)\max_{0\le d\le D}\{\|{\cl P}_d(x)\|\}.
  \end{equation}
  It will be convenient for us to extend the above definition of $kh$ as follows:
  for any $D$ and any $x \in W_{\le D}$ we define
   \begin{equation}\label{kh} kh(x)=\sup_{0\le d\le D} kh({\cl P}_d(x)).\end{equation}
  
  Combining \eqref{as4} with \eqref{rx0} we now obtain that
  there are constants $c>0$ and $\alpha>0$
such that \begin{equation}\label{rx}
\forall d\ \forall x\in W_{\le d}  \quad  (c d^\alpha)^{-1}kh(x)\le  \|x\| \le c d^\alpha kh(x).
\end{equation}

\subsection{}\label{f}
Returning to the situation of  Theorem \ref{t3},
let $H^{(k)}_i=L_2(\phi_i^{(k)})$ with distinguished vector $\xi_i^{(k)}$,
GNS representation $\pi^{(k)}_i: {\cl A}_i\to A^{(k)}_i \subset B(H^{(k)}_i)$ with $A^{(k)}_i=\overline{\pi^{(k)}_i( {\cl A}_i)}$. 
We define $A^{(k)}=*_{i\in I} A^{(k)}_i$,  $H^{(k)}=*_{i\in I} H^{(k)}_i$ and let $\pi^{(k)}:\ {\cl A} \to *_{i\in I} A^{(k)}_i\subset B(H^{(k)})$ be the corresponding representation. 
Let
$H^{\U}$ (resp. $H^\U_i$) denote the ultraproduct of $(H^{(k)})$  (resp. $(H^{(k)}_i)$).\\
 We will use \ref{e} when $\phi=*_{i\in I}\phi_i$
and $\phi^{(k)}=*_{i\in I}\phi^{(k)}_i$.
We denote by 
$\pi:\ {\cl A}=*_{i\in I}{\cl A}_i \to B(H)$  the GNS representation relative to $\phi$. 
We have  natural identifications
$$(L_2(\phi),\xi_\phi)=*_{i\in I}(H_i,\xi_i)\quad{\rm and }
\quad (L_2(\phi^{(k)}),\xi_{\phi^{(k)}})=*_{i\in I}(H^{(k)}_i,\xi^{(k)}_i).$$
If we assume that $\phi_i^{(k)}\to \phi_i$ pointwise, then  
$\phi^{(k)}\to \phi$ pointwise on $\cl A$ and of course $\phi=\phi^\U$ on $\cl A$. Therefore
 the correspondence $ b \xi\mapsto  b  \xi^\U=[  b  \xi^{(k)}]_\U$
 is isometric  from $H$ to $H^\U$ (see \ref{x1}). 
Similarly, 
the correspondence $  \xi^{op}b \mapsto   \xi^\U b=[  \xi^{(k)} b]_\U$
  is isometric  from  from $H^*$ to ${H^\U}^*$ (see \ref{x1}).
  We will denote  respectively by $V:\ H \to H^\U$ and
   $W:\ H^* \to {H^\U}^*$   these \emph{isometric  embeddings}, so that   we have
 for any $b\in \cl A$  
  \begin{equation}\label{as10}V(b \xi )=[b\xi^{(k)} ]_\U \in {H^\U}\quad{\rm and}\quad
W(\xi b)=[\xi^{(k)} b ]_\U\in {H^\U}^*.\end{equation}
Assume now that $\phi_i^{(k)}\stackrel {s}{\to} \phi_i $.
Then  (see \ref{x10}) we also have an \emph{isometric  embedding}
$$\psi_i    :\ A_i\to A_i^\U\subset B(H_i^\U)$$
such that $\psi_i(\pi_i(b))=\pi_i^\U(b)$ for any $b\in {\cl A}_i$.

 \begin{proof}[Proof of Theorem \ref{t3}]
 Let $\phi= \ast_{i\in I}\phi_i $ and $\phi^{(k)}=\ast_{i\in I}\phi^{(k)}_i$. 
 Recall ${\cl A}=*_{i\in I} {\cl A}_i$ (algebraic free product). Note that ${\cl A}=\cup_K \cl W_{\le K}$.
 The pointwise convergence on ${\cl A}$ of 
 $\phi^{(k)}$ to $\phi$   is obvious by definition of the   free product of states.
 To show the strong convergence it suffices to show
 that $\lim\nolimits_\U \|\pi^{(k)}(b)\|=\|\pi(b)\|$ for any $b\in {\cl A}$ and any 
 non trivial ultrafilter $\U$ on $\N$.

 The main point is that, by the Ricard-Xu inequality, 
there are   constants $c>0$ and $\alpha>0$
such that  
 \begin{equation}\label{5}\forall b\in \cl W_{\le d}  \quad  
 (c d^\alpha)^{-1}\lim\nolimits_\U\|\pi^{(k)}(b)\| \le  \|\pi(b)\| \le c d^\alpha\lim\nolimits_\U\|\pi^{(k)}(b)\|.\end
 {equation}
If we accept this result, the proof is immediate:
we just note that   $(b^*b)^m$ is of degree at most $2md$,
therefore 
$(c (2md)^\alpha)^{-1} \lim\nolimits_\U\|\pi^{(k)}((b^*b)^m)\|\le  \|\pi((b^*b)^m)\|=\|\pi(b)\|^{2m} \le c (2md)^\alpha \lim\nolimits_\U\|\pi^{(k)}((b^*b)^m)\|$ and $\|\pi^{(k)}((b^*b)^m)\|=\|\pi^{(k)}( b)\|^{2m}$. 
So we find
$$(c (2md)^\alpha)^{-1/2m} \lim\nolimits_\U\|\pi^{(k)}(b)\|\le   \|\pi(b)\| \le (c (2md)^\alpha)^{1/2m}\lim\nolimits_\U\|\pi^{(k)}(b)\|$$
and letting $m\to \infty$ yields the equality $\lim\nolimits_\U\|\pi^{(k)}(b)\|=\|\pi (b)\|$.

We now turn to the proof of \eqref{5}.
By  the Ricard-Xu inequality \eqref{rx}, we have 
\begin{equation}\label{6}\forall K\ge1, \forall b\in \cl W_{\le K}  \quad  (c d^\alpha)^{-1}kh(\pi(b))\le  \|\pi(b)\| \le c d^\alpha kh(\pi(b)).\end
 {equation}
and 
 $$(c d^\alpha)^{-1}\lim\nolimits_\U kh( \pi^{(k)}(b) )\le  \lim\nolimits_\U\|\pi^{(k)}(b)\| \le c d^\alpha \lim\nolimits_\U kh(\pi^{(k)}(b)  ).$$
 Thus to conclude, it suffices to show that $\lim\nolimits_\U kh(\pi^{(k)}(b))=kh(\pi(b))$ for any $b\in \cl W_{\le K} $ and any $K\ge 1$.
 
 Let $b=\sum_0^K b_d$
 and $b=\sum_0^K b^{(k)}_d$ be the decomposition of ${\cl W}_{\le K} $
 into its homogeneous parts relative respectively to $\phi$
 and $\phi^{(k)}$.
 More precisely, $b_d={\cl P}_d(b)$ and $b^{(k)}_d={\cl P}^{(k)}_d(b)$.
  By \eqref{as56} we have
 $$v^{(k)} b_d= {\cl P}^{(k)}_d( v^{(k)} b)$$
 and hence by \eqref{as55} 
  \begin{equation}\label{x4}\| v^{(k)} b_d- b^{(k)}_d\|\to 0 \end{equation}
 with respect to the maximal $C^*$-norm on ${\cl A}$.

 Let us denote by  $t^{(k)}_r,s^{(k)}_r$   
 the mappings $t_r,s_r$ relative to the free product $\phi^{(k)}$.
 
 By \eqref{rx0+} and \eqref{kh} to conclude it obviously suffices to show
 $$\lim_{k\to \infty} \|t^{(k)}_r (b_d^{(k)}) \|=  \|t_r (b_d) \|
 \ {\rm and} \ \lim_{k\to \infty} \|s^{(k)}_r (b_d^{(k)}) \|=  \|s_r (b_d) \|.$$
 By \eqref{x4}, it actually suffices to show
\begin{equation}\label{x5}\lim_{k\to \infty} \|t^{(k)}_r (v^{(k)} b_d) \|=  \|t_r (b_d) \|
 \ {\rm and} \ \lim_{k\to \infty} \|s^{(k)}_r (v^{(k)} b_d) \|=  \|s_r (b_d) \|.\end{equation}

Fix $d$. We may assume
 that $b_d$  is a finite sum of the form 
 $ b_d =\sum^N_{\alpha=1} b(\alpha)$ with $b(\alpha) $ of the form
   $ b(\alpha)= x_1(\alpha)\cdots x_d(\alpha)$
 where  $x_1(\alpha)\in \stackrel{\circ}{{\cl A}_{i_1}} ,\cdots ,x_d(\alpha)\in \stackrel{\circ}{{\cl A}_{i_d}}$ and    $i_1\not=i_2\not= \cdots\not= i_d $.  
 
We remind the reader that   $\xi$ denotes the distinguished unit vector in $H=*_{i\in I}  H_i$, and \\
 $$t_r(b_d)=\sum\nolimits_\alpha x_1(\alpha)\cdots x_r (\alpha) \xi  \otimes \xi  x_{r +1}(\alpha) \cdots  x_d(\alpha)\in H\otimes H^*$$
 is viewed as an element of $B(H)$. Moreover,  in the   case
 $I=\{1,2\}$,
 $s_r(b_d)\in B(H) \otimes [  A_1 \oplus  A_2   ]$.\\ 
 By  \eqref{x3}, we have 
 \begin{equation}\label{x6} v^{(k)}(b_d)=\sum\nolimits_\alpha v^{(k)}(x_1(\alpha)\cdots x_d(\alpha))
 =v^{(k)}(x_1(\alpha))\cdots v^{(k)}(x_d(\alpha))\end{equation}
 and also for any $1\le j\le d$ we have $v^{(k)}(x_j(\alpha)) =x_j(\alpha) - \phi^{(k)}(x_j(\alpha)) 1$ and hence
 \begin{equation}\label{x7}\| v^{(k)}(x_j(\alpha)) - x_j(\alpha)\|\to 0 \end{equation}
 where the norm is (say) the maximal $C^*$-norm on ${\cl A}$.

 Recall   that $H^\U$
 denotes the Hilbert space ultraproduct of the free products defined by $(H^{(k)},\xi^{(k)} )=
 *_{i\in I}( H^{(k)}_i,\xi^{(k)}_i)$. 
We should compare $t_r(b_d)$ with $t^{(k)}_r(v^{(k)} b_d)\in  H^{(k)}\otimes {H^{(k)}}^*.$
By \eqref{x6} and \eqref{x7}  we have
 \begin{equation}\label{x8}\| t^{(k)}_r(v^{(k)} b_d)-\sum\nolimits_\alpha x_1(\alpha)\cdots x_{r}(\alpha) \xi^{(k)}  \otimes \xi^{(k)}  x_{r +1}(\alpha) \cdots  x_d(\alpha)\|_{  B(H^{(k)}) } 
 \to 0.\end{equation}
 Let
   $$ T_r^{(k)} =\sum\nolimits_\alpha x_1(\alpha)\cdots x_{r}(\alpha) \xi^{(k)}  \otimes \xi^{(k)}  x_{r +1}(\alpha) \cdots  x_d(\alpha),$$
      $$ T_r  =\sum\nolimits_\alpha [x_1(\alpha)\cdots x_{r}(\alpha) \xi^{(k)}  ]_\U \otimes [\xi^{(k)}  x_{r +1}(\alpha) \cdots  x_d(\alpha)]_\U.$$
   By  \eqref{x12} we have
  \begin{equation}\label{9}\|T_r\|_{  B(H^\U) }=\lim\nolimits_\U\|T_r^{(k)}  \|_{  B(H^{(k)}) }.\end{equation}
Consider now (see \ref{f}) the \emph{isometries}  $V:\ H \to H^\U$ and 
$W:\ H^* \to {H^\U}^*$.
  We have then,   by \eqref{as10}
  $$(V\otimes W)(t_r(b_d))=T_r$$
  from which  $\|T_r\|_{  B(H^\U) }=\|t_r(b_d)\|_{  B(H) }$ follows. Now by \eqref{x8} and \eqref{9}
  we conclude that $$\|t_r(b_d)\|_{  B(H) }= \lim\nolimits_\U \| t^{(k)}_r(v^{(k)} b_d)  \|_{  B(H^{(k)}) } .$$ So far we used only the pointwise convergence.

Similarly, assuming $I=\{1,2\}$ for simplicity, 
we  should compare  $s_r(b_d)$ with  $s^{(k)}_r(v^{(k)} b_d)$.
We set
   $$ S_r^{(k)} =\sum\nolimits_\alpha x_1(\alpha)\cdots x_{r-1}(\alpha) \xi^{(k)}  \otimes \xi^{(k)}  x_{r +1}(\alpha) \cdots  x_d(\alpha) \otimes J^{(k)}\pi^{(k)}v^{(k)}( x_r(\alpha)),$$
   
    $$ S_r  =\sum\nolimits_\alpha [x_1(\alpha)\cdots x_{r-1}(\alpha) \xi^{(k)}]_\U  \otimes [\xi^{(k)}  x_{r +1}(\alpha) \cdots  x_d(\alpha)]_\U  \otimes [J^{(k)}\pi^{(k)}v^{(k)}( x_r(\alpha))]_\U.$$
   By \eqref{x6} and \eqref{x7}  we have
$$\| s^{(k)}_r(v^{(k)} b_d)-S_r^{(k)}  \|_{  B(H^{(k)})\otimes_{\min}({A^{(k)}_1}\oplus {A^{(k)}_2})  } 
 \to 0,$$
 and hence
  \begin{equation}\label{x9}
  \lim\nolimits_\U \| s^{(k)}_r(v^{(k)} b_d)\|_{  B(H^{(k)})\otimes_{\min}({A^{(k)}_1}\oplus {A^{(k)}_2})  } 
= \lim\nolimits_\U \|S_r^{(k)}  \|_{  B(H^{(k)})\otimes_{\min}({A^{(k)}_1}\oplus {A^{(k)}_2})  } 
  .\end{equation}
   By \eqref{ida2}
  \begin{equation}\label{70} \|S_r\|_{B({H^{\U}}  ) \otimes_{\min}({A^\U_1}\oplus {A^\U_2})}=\lim\nolimits_\U\|S_r^{(k)}\|_{B({H^{(k)}}  ) \otimes_{\min}({A^{(k)}_1}\oplus {A^{(k)}_2})}. \end{equation} 
  Now since we assume  \emph{strong} convergence, 
  as explained in \ref{f} we have \emph{isometric embeddings}
  $$  \psi_1   :\ A_1\to  A^\U_1  \text{  and }  \psi_2   :\ A_2\to  A^\U_2$$
  such that
  $$S_r=    (V \otimes W) \otimes [\psi_1\oplus \psi_2] (s _r(b_d))$$
  from which $\|S_r\|=\|s_r(b_d)\|$ follows.
   Thus, by \eqref{x9} and \eqref{70}, we obtain \eqref{x5},
 and this concludes
the proof. 
\end{proof}
 
 We now turn to the situation considered in the abstract.
  \begin{cor}
   Let $(A_1^{(k)},\phi_1^{(k)})$, $(A_2^{(k)},\phi_2^{(k)})$, 
   $(A_1 ,\phi_1)$ and $(A_2 ,\phi_2)$ be faithful $C^*$-probability spaces.
 Let $\{X^{(k)}_m\mid m\in {\cl I} \}\subset A_1^{(k)} $, 
 $\{Y^{(k)}_n\mid n\in {\cl J} \}\subset A_2^{(k)} $ and 
 $\{X_m\mid m\in {\cl I} \}\subset A_1$, $\{Y_n\mid n\in {\cl J} \}\subset A_2$   be    families
 of non-commutative random variables, generating
 respectively $A_1^{(k)},A_2^{(k)},A_1,A_2$.
Assume $\{X^{(k)}_m\mid m\in {\cl I} \}  \stackrel{s}{\to} \{X_m\mid m\in {\cl I} \}$
 and  $\{Y^{(k)}_n\mid n\in {\cl J} \} \stackrel{s}{\to} \{Y_n\mid n\in {\cl J} \}$
 when $k\to \infty $. Assume moreover that, for each $k$, 
 $\{X^{(k)}_m\mid m\in {\cl I} \}$ and $\{Y^{(k)}_n\mid n\in {\cl J} \}$ are $*$-free and also 
 that $\{X_m\mid m\in {\cl I} \}$
 and  $\{Y_n\mid n\in {\cl J} \}$ are $*$-free.
 Then the joint family $\{X^{(k)}_m,Y^{(k)}_n\mid m\in {\cl I}, n\in {\cl J}\}$
 viewed as sitting in the free product $(A_1^{(k)},\phi_1^{(k)})\ast  (A_2^{(k)},\phi_2^{(k)})$
 tends strongly to $\{X_m,Y_n\mid m\in {\cl I}, n\in {\cl J}\}$,
 viewed as sitting in the free product $(A_1 ,\phi_1 )\ast  (A_2 ,\phi_2 )$.
\end{cor}
 \begin{proof}
   Let ${\cl A}_1$ (resp. ${\cl A}_2$) be the unital $*$-algebra generated by 
   non-commutative (i.e. algebraically free) variables 
   $\{x_m,x_m^*\mid m\in {\cl I} \}$ 
 (resp. $\{y_n,y_n^*\mid n\in {\cl J} \}$).
We define the associated state $\phi_1$ on ${\cl A}_1$ by $\phi_1(P)=\phi (P(X_m,X_m^*))$ and similarly for
$\phi_2$ on ${\cl A}_2$. Repeating this for each $k$,
 this leads to states $\phi^{(k)}_i$ on ${\cl A}_i$ for $i=1,2$. We may  
 identify the free product ${\cl A}_1\ast {\cl A}_2$ with the  unital $*$-algebra generated by 
   non-commutative variables 
   $\{x_m,x_m^*,y_n,y_n^* \mid m\in {\cl I}, n\in {\cl J}\}$. Since the GNS
   representations are assumed isometric,
   the Corollary   appears as a particular case of the preceding Theorem.
   \end{proof}
   
   \begin{rem} If a state $\phi$ is faithful on $A$, then
   its associated GNS representation  is faithful. Moreover,
   the GNS representation associated to
   the restriction of  $\phi$ to any unital $C^*$-subalgebra of $A$
  is still faithful, so the requirement that the variables generate
   the $C^*$-algebras can be  dispensed with if we assume all states faithful.
   Moreover, if $\phi$ is faithful on $A$, then for any $x\in A$ we have
   $$\|\pi_\phi(x)\|=\lim_{p\to \infty} \uparrow (\phi ((x^*x)^p ))^{1/2p}.$$
   
   \end{rem}
   
\subsection{}\label{rem2}  With the notation of the Corollary, whenever 
$X^{(k)}=\{X^{(k)}_m, {X_m^{(k)}}^*\mid m\in {\cl I} \}$ on $(A^{(k)},\phi^{(k)})$ converges in moments to  $\{X_m,{X_m}^*\mid m\in {\cl I} \}$ on $(A,\phi)$,
it is well known that
for any non-trivial ultrafilter $\U$ and any polynomial $P$, we have
 \begin{equation}\label{11}  \|P(X )\| \le \lim\nolimits_\U \|P(X^{(k)}  )\|.\end{equation}
Indeed,  
for any $\vp>0$ there are polynomials $Q,R$  with unit norm
in $L_2(\phi)$ so that
$$\|P(X )\| -\vp< |\phi(P(X)Q(X)R(X))|=\lim\nolimits_\U |\phi^{(k)}(P(X^{(k)})Q(X^{(k)})R(X^{(k)}))|\le \lim\nolimits_\U \|P(X^{(k)} )\|,$$
from which \eqref{11} follows.
The converse inequality is the essence of strong convergence.
\section{Commutative case. Random matrices}\label{RM}

 With the same notation as \ref{rem2}, if the variables 
$X=\{X_m, {X_m}^*\mid m\in {\cl I} \}$ all commute
 i.e.  we are dealing with a family $ \{X_m \mid m\in {\cl I} \}$ of commuting normal operators in $B(H)$,  then
 the (commutative)  $C^*$-algebra $A$ they generate in $B(H)$  is isometric
to the $C^*$-algebra $C(K)$ of all continuous functions on a compact set $K$, namely the spectrum of $A$. Moreover, $\phi$ corresponds
to a probability measure on $K$. Assume  ${\cl I}$ finite for simplicity.
Then the situation reduces to the following:
We have a probability measure $\mu$ with support  a compact set
$K\subset \CC^{\cl I}$
and $X_m\in C(K)$ is defined by $X_m(\lambda)=\lambda_m$ for all
$\lambda=(\lambda_m)\in K$.
For any $a\in \cl A$,   $\phi(a)=\int a d\mu$ and $\pi_\phi(a)$ is the operator of multiplication by $a$ on $L_2(\mu)$.
When the family  $(X_m)$   is reduced to a single normal operator $X$,
$K$ is the spectrum of $X$.

Similarly, if all $\{X^{(k)}_m, {X^{(k)}_m}^*\mid m\in {\cl I} \}$ commute, 
we may reduce consideration  to  $A^{(k)}=C(K^{(k)})$
with $\phi^{(k)}=\mu^{(k)}$
for some $K^{(k)}\subset \CC^{\cl I}$
and some probability  $\mu^{(k)}$ with support  $K^{(k)}$,
and again $X^{(k)}_m(\lambda)=\lambda_m$ for all
$\lambda\in K^{(k)}$.

Then $X^{(k)}\stackrel{s}{\to} X$ implies that
for any polynomial $f$ in $\lambda_m, \bar \lambda_m$ we have
  \begin{equation}\label{x34}\sup_{\lambda\in K^{(k)}} |f(\lambda)| \to \sup_{\lambda\in K} |f(\lambda)|.\end{equation}
Taking $f(\lambda)=\lambda_m$ we see that 
the sets $\cup_k  K^{(k)}\subset \CC$  and $K$ are all included
in some compact set $L\subset \CC^{\cl I}$. By the Stone-Weierstrass theorem on $L$,
\eqref{x34} remains valid for any continuous function on $L$.
Applying this to $f(\lambda)=d(\lambda,K)$ we find
 \begin{equation}\label{x33}\lim_{k\to \infty}\sup_{\lambda\in K^{(k)}} d(\lambda, K)= 0.\end{equation}

  The mere convergence in moments $X^{(k)}{\to} X$ is equivalent
  (by  Stone-Weierstrass on $L$)
to the weak convergence $\mu^{(k)}{\to} \mu$. 
In the  language of probability, the latter means that
$X^{(k)}{\to} X$ ``in distribution", or ``in law".
 Thus
$X^{(k)}\stackrel{s}{\to} X$ iff $\mu^{(k)}{\to} \mu$ weakly and \eqref{x33} holds.

Wigner's classical theorem about the convergence of the eigenvalues
of Gaussian   random matrices was strengthened in \cite{HT} as follows:
let $X^{(k)}(\omega)$ be a    random $k \times k$-matrix the entries of which are 
 independent  complex Gaussian  $N(0,k^{-1})$, then for almost all $\omega$, the \emph{nonnormal} random matrices 
$X^{(k)}(\omega)$ tend strongly to a circular random variable.
A similar result is valid for the classical Gaussian Wigner Hermitian (and hence normal) matrices
(model for the so-called GUE) now with  a semi-circular limit. 
 In the latter case, if 
 $\mu^{(k)}(\omega)$ is the spectral probability distribution of the eigenvalues of $X^{(k)}(\omega)$, with support $K^{(k)}(\omega)$, then
  $\mu^{(k)}(\omega)$ tends weakly   to the ``circular" probability
   measure  on $K=[-2,2]$ and
 \eqref{x33} holds for almost all $\omega$. A similar result
 holds for random unitary $k \times k$-matrices. In that case
 the limit $\mu$ is the uniform Haar probability on the unit circle.
In Random Matrix Theory,    \eqref{x33}  for the spectra of random matrices, is viewed as a result on ``the edge of the spectrum",
while mere weak convergence deals with 
``the bulk of the spectrum".\\
See  \cite{HT,CM,HST,Sc,A} for more general results.\\

\section{Strong convergence and operator spaces}\label{OS}

Let $( A , \phi )$ and $( A^{(k)}, \phi^{(k)})$ be  $C^*$-probability spaces,
with $\phi$ and $\phi^{(k)}$ all   faithful.
For any $N\ge 1$, we equip $M_N(A)$ 
with the faithful  state $\phi\otimes \tau_N$ where $\tau_N$ is the normalized trace
on $M_N$
(and similarly for  $M_N(A^{(k)})$).
Let $\{X _m\mid m\in {\cl I} \}\subset A $  and  $\{X^{(k)}_m\mid m\in {\cl I} \}\subset A^{(k)}$. \\ If $\{X ^{(k)}_m\mid m\in {\cl I} \}\stackrel{s}\to \{X _m\mid m\in {\cl I} \}$, 
then for any $N$,   any $a\in M_N$ and any finitely supported
family $\{a _m\mid m\in {\cl I} \}\subset M_N $, we have
\begin{equation}\label{x35p}1\otimes a+\sum\nolimits_m   X ^{(k)}_m\otimes a_m \stackrel{s}\to 1\otimes a+\sum\nolimits_m   X _m\otimes a_m .\end{equation} 
In particular we have for any such $N,a,a_m$ and for any non trivial $\U$
 \begin{equation}\label{x35}\lim_\U \|1\otimes a+\sum\nolimits_m   X ^{(k)}_m\otimes a_m \|_{M_N(A^{(k)})}=
\|1\otimes a+\sum\nolimits_m   X_m\otimes a_m \|_{M_N(A )}.
\end{equation}
Indeed, the strong convergence implies that the mapping
 $u:\ P(X _m)\mapsto P(X^\U _m)$ is an isometric, 
 and hence   completely isometric, $*$-homomorphism.
 Therefore if we restrict $u$ to the linear span, denoted by $E$,
 of the unit and $\{X _m\mid m\in {\cl I} \}$ then we obtain 
 \eqref{x35}. Let $S=1\otimes a+\sum\nolimits_m   X_m\otimes a_m$,
 and let $P$ be a polynomial in noncommuting variables $x,x^*$.
 Then $P(S)=  1\otimes a'+\sum\nolimits_m   P_m(X_m)\otimes a'_m$
 for some $*$-polynomials $(P_m)$ and some  $a',a' _m\in M_N$.
 Clearly if $ (X^{(k)}_m)\stackrel{s}\to  (X_m)$
 then $(P_m(X^{(k)}_m))\stackrel{s}\to (P_m(X_m))$.
 Therefore, \eqref{x35} applied with
      $(P_m(X^{(k)}_m))$ and $(P_m(X_m))$ 
  in place of    $X^{(k)}_m$ and $X_m$  gives us \eqref{x35p}.
 
 In the converse direction, we have the following two ``linearization tricks" :
 \begin{pro}[\cite{Pi}] If the operators $\{X ^{(k)}_m\mid m\in {\cl I} \}$ converge  in moments and are all  unitary, then  \eqref{x35} implies $\{X ^{(k)}_m\mid m\in {\cl I} \}\stackrel{s}\to \{X _m\mid m\in {\cl I} \}$.
 \end{pro} \begin{proof}
 Indeed, by \cite[Prop.1.7]{Pi}, if  $u_{|E}$   is completely contractive, then $u:\  P(X _m)\mapsto P(X^\U _m)$
 is contractive. By  \ref{rem2}, since we assume
 convergence in moments, $u$ is isometric.
   \end{proof}
  \begin{pro}[\cite{HT}] 
If the operators $\{X ^{(k)}_m\mid m\in {\cl I} \}$ converge  in moments and are all Hermitian, then   \eqref{x35p} 
 restricted to Hermitian matrices $a,a_m$ implies that  $\{X ^{(k)}_m\mid m\in {\cl I} \}\stackrel{s}\to \{X _m\mid m\in {\cl I} \}$.
 \end{pro}
  \begin{proof}
 Indeed, let $S^{(k)}=1\otimes a+\sum\nolimits_m   X ^{(k)}_m\otimes a_m$.
 If $S ^{(k)}$ is Hermitian, and $S^{(k)}\stackrel{s}\to S$ then
 \eqref{x33} holds, $K^{(k)}$ and $K$ being the spectra
 of $S^{(k)}$ and $S$.  This implies
 that the spectrum of $S ^{\U}$ is included in that of $S$. Thus the Proposition follows from \cite[Th.2.2]{HT} (recalling \ref{rem2}).
 \end{proof}
\section{Amalgamated free products}

Since the Ricard-Xu version of the Khintchine inequalities extends
to reduced free products with amalgamation (see \cite[\S 5]{RX}), it is not surprising
that the preceding approach also does, as was kindly pointed out to the author independently by \'Eric Ricard and Paul Skoufranis.

We should first define what is meant by strong convergence
in this framework.

Let $A$ be a unital $C^*$-algebra, with a unital $C^*$-subalgebra
$D\subset A$.
Let 
$\phi:\ A\to D$  be   a conditional expectation. 
As in \cite[p. 138]{BO} we denote by $L_2(\phi)$
 the  right Hilbert $D$-module obtained by the classical
 GNS construction. We refer to \cite{La,BLM} for more information on Hilbert $C^*$-modules. The
 $D$-valued inner product of two elements
 $\dot a,\dot b$ of $L_2(\phi)$ associated to
$a,b\in A$ is defined by $\langle \dot a,\dot b \rangle=\phi(a^*b)$.
We have $\| \dot a \|_{L_2(\phi)}=\| \langle \dot a,\dot a \rangle \|_D^{1/2}$.
We will denote by $\pi_\phi:\ A\to {\bb B}(L_2(\phi))$ the
associated representation of $A$ in the $C^*$-algebra ${\bb B}(L_2(\phi))$ of adjointable
maps on $L_2(\phi)$, so that $\pi_\phi(a) \dot b=\dot{ab}$. 
Let $\xi_\phi=\dot 1\in L_2(\phi)$ so that $\dot{a} =\pi_\phi(a) \xi_\phi$. 
Then $\pi_\phi(A) \xi_\phi$ is dense in $L_2(\phi)$,
and $\phi(a)=\langle  \xi_\phi, \pi_\phi(a) \xi_\phi\rangle$. 
We will denote simply $\pi_\phi(a) \xi_\phi=a \xi_\phi $.

We say that a  conditional expectation $\phi$ is nondegenerate
if $\pi_\phi$ is faithful (i.e. isometric).
We should warn the reader
that although this (well established) notation does not reflect it,
the preceding notions obviously all depend on $D$.

 Let $H=L_2(\phi)$. By definition any adjointable $T:\ H \to H$ admits
 an adjoint $T^*:\ H \to H$ such that, as usual, $\langle T^* y, x\rangle =\langle   y, Tx\rangle
 $.\\
Let   $H^*={\bb B}(H,D)$.
Clearly, $H^*$ is antilinearly isomorphic to $H$, via the correspondence
$h\in H \mapsto t_h\in H^*$ defined by $t_h(y)= \langle h, y\rangle$.
The space $H^*$ is equipped with the left Hilbert $D$-module structure
defined by $  d\cdot  t_h= t_{h d^*} $, and $D$-valued inner product
$\langle t_k, t_h\rangle= \langle h, k\rangle$ ($h,k\in H$).
Let $\xi^{op}\in H^*$ be defined by $\xi^{op} (h)=\langle \xi,h\rangle=\phi(h)$.
We have a natural
embedding $ { A}^{op} \subset H^*$
 that we will write
as $a\mapsto \xi^{op} a$, with the notation $(fa  )(x)=f(ax)$ when
$f\in H^*$, $x\in H$, $a\in A$. Then   $$\| \xi^{op} a\|_{H^*}=\|\phi(aa^*)\|_D^{1/2},$$ and $\xi^{op} A$ is dense in $H^*$.
\\ We have  a natural representation
$\pi^{op}_\phi:\ { A}^{op}\to {\bb B}( H^*)$, defined as before: 
$\pi^{op}_\phi(a)$ is the transpose of the adjointable map $\pi_\phi(a)$,
i.e. $\pi^{op}_\phi(a)= {}^t \pi_\phi(a)$, 
in the following sense: The operator ${}^t T:\ H^*\to H^*$
is characterized by the identity $ {}^t T (\xi^{op} a) (x)=\phi ( aT (x))$ ($a\in A,x\in H$). \\Moreover $\|{}^t T\|_{{\bb B}( H^*)}=\|T\|_{{\bb B}( H)}$.
Equivalently, if we define $\bar h\in H^*$ 
by $\bar h (x)= \langle h,x\rangle$, i.e. we set $\bar h=t_h$,  then we may write
$${}^t T(\bar h)= \overline {T^* (h)}.$$

Let ${\cl A} $ be a dense unital  $*$-subalgebra of a $C^*$-algebra $B$
such that $D\subset {\cl A}\subset B$.
We will say that a mapping $ \phi :\ {\cl A}\to D$ is a  conditional expectation
if it extends to a   conditional expectation on $B$ with range $D$.  
The GNS construction
produces a Hilbert $D$-module $H=L_2(\phi)$ with distinguished unit vector $\xi_\phi$
and a $*$-homomorphism $\pi_\phi:\ 
{\cl A} \to \B(H)$ such that $\phi(x)=\langle \xi_\phi,\pi_\phi(x)\xi_\phi\rangle$
for any $x\in \cl A$.
\\ Let $A=\overline{\pi_\phi(\cl A)}$. Then
for any $d\in D$, $ {\pi_\phi(d)} $ is the left action of $d$
on $H$. Thus
 $ {\pi_\phi(D)} $ is a copy
of $D$ in $A$ and $\psi:\ x\mapsto \pi_\phi (\langle \xi_\phi,x\xi_\phi\rangle)$ is a  conditional expectation from $A$ onto $ {\pi_\phi(D)} \simeq D$.

We will say that a sequence $ \phi^{(k)} :\ {\cl A}\to D$  ($k\in \N$)  
 of   conditional expectations tends strongly to $\phi:\ {\cl A}\to D$,
 if it converges pointwise to $\phi$ and if moreover
we have $\|\pi_{\phi^{(k)}}(a) \|\to \|\pi_{\phi}(a) \|$ for any $a\in {\cl A}$.
We denote this again  by  $ \phi^{(k)}  \stackrel{s}\to \phi $.

Let $B_i$ ($i\in I$) be a family of unital $C^*$-algebras.
Let  ${\cl A}_i\subset B_i$ be dense  $*$-subalgebras, 
 each ${\cl A}_i $ containing a unital copy of a fixed $C^*$-algebra $D$, given with     conditional expectations $ \phi_i :\ {\cl A}_i\to D$.
 Let  $\cl A$ denote the $*$-algebra that is the algebraic  free product, amalgamated over $D$,
 of the family $({\cl A}_i)_{i\in I}$. 
 The   free product
of the family of conditional expectations $ \phi_i :\ {\cl A}_i\to D$ will be defined below 
as  a conditional expectation $ \phi :\ {\cl A}\to D$, but
we need more notation to make this clear, in part because
the free product requires  nondegenerate  conditional expectations.

Let $H_i=L_2(\phi_i)$ and $\xi_i=\xi_{\phi_i}$.
Let $A_i=\ovl{\pi_{\phi_i } ({\cl A}) } \subset \B(H_i)$.
Then $D$ can be identified with $\pi_{\phi_i } (D) \subset A_i$ isometrically, and
  $\psi_i(x) =\pi_{\phi_i }(\langle \xi_i,x\xi_i\rangle)$
 is a \emph{nondegenerate}  conditional expectation from $A_i$ onto $\pi_{\phi_i }(D)\simeq D$.

 We refer the reader to \cite[p. 138]{BO} for the precise definition of the 
reduced free product $A=\ast_{D}(A_i,\psi_i)$ of the  family $(A_i)_{i\in I}$ with respect to \emph{nondegenerate} conditional expectations
$ (\psi_i)_{i\in I} $: One first introduces the
right
Hilbert $D$ module that is the  free product $(H,\xi)$
of the family $ (H_i,\xi_i)$, and the $*$-homomorphism
$\pi: \cl A \to {\bb B} (H)$ 
associated to the 
algebraic free product of the  family of morphisms
 $\pi_{\phi_i}: \cl A_i \to A_i$
(acting on $H$ on the left). Let  $A=\ovl{\pi(\cl A)}\subset {\bb B} (H)$. Lastly 
  $\ast_{i\in I} \phi_i:\ {\cl A}\to D$ is defined by   $$(\ast_{i\in I} \phi_i)(a)=   \langle  \xi, \pi(a) \xi\rangle.$$ 
  This  is  a    conditional expectation  
extending each $\phi_i$.


\begin{thm}\label{t3a} Let $D\subset {\cl A}_i $ ($i\in I$) be a family of unital $C^*$-algebras, as above.
 Let $\{\phi^{(k)}_i\mid i\in I\}$  ($k\in \N$) be a sequence of families of mappings,
 each $\phi^{(k)}_i$ being a    conditional expectation from
  ${\cl A}_i$ onto $D $.  
 Assume that  
 we have   conditional expectations $\phi_i:\ {\cl A}_i\to D$  such that, 
 for each $i\in I$,   
 $ \phi^{(k)}_i \stackrel{s}\to \phi_i$   when $k\to \infty$.
 Then $$*_{i\in I}\phi^{(k)}_i \stackrel{s}\to *_{i\in I}\phi_i .$$
 \end{thm}

In the  amalgamated case, 
\eqref{rx0+} remains valid provided the norms of $s_r$ and $t_r$ are interpreted
as follows. 
Assume $I=\{1,2\}$ for simplicity.
As before $\cl A=\cup_d {\cl W}_{\le d}$ with the same meaning of 
${\cl W}_{\le d}$ or ${\cl W}_d$.

Let  $b \in  \cl A$. Let $b=\sum b_d$ be its decomposition
into homogeneous terms.
Fix $d$. We may assume
 that $b_d$  is a finite sum of the form 
 $ b_d =\sum^N_{\alpha=1} b(\alpha)$ with $b(\alpha) $ of the form
   $ b(\alpha)= x_1(\alpha)\cdots x_d(\alpha)$
We then set as before
  \begin{equation}\label{x13} t_r(b_d)=\sum\nolimits_\alpha x_1(\alpha)\cdots x_r (\alpha) \xi  \otimes \xi^{op}   x_{r +1}(\alpha) \cdots  x_d(\alpha)\in H\otimes H^*\end{equation}
and
 \begin{equation}\label{x14}  
 s_r(b_d)= \sum\nolimits_\alpha x_1(\alpha)\cdots x_{r-1} (\alpha)  \xi \otimes  \xi^{op}  x_{r +1}(\alpha) \cdots  x_d(\alpha)\otimes J(x_r)\in H\otimes H^*\otimes [A_1\oplus A_2].\end{equation}
Then $\|t_r(b_d)\|$ is the norm in $\bb B (H)$
and $\|s_r(b_d)\|$ is its norm in ${\bb B} (H)\otimes_{\min} [A_1\oplus A_2].$

With this reinterpretation, \eqref{x11}
remains valid (see \cite[\S 5]{RX}).
Thus, with the obvious extension of the previous notation, the proof of Theorem \ref {t3a} boils down
to check that we still have \eqref{x5}.
One difficulty is that Hilbert modules do not necessarily
admit orthonormal bases, so that we cannot argue
as we did above to check \eqref{x12}. Instead,
following Ricard and Xu in \cite{RX}, we 
restrict without loss of generality to
the case of separable $C^*$-algebras and
a countable set of indices $I$. 
To justify this reduction to the separable case,
observe that for any countable subset $S\subset A$
in a unital $C^*$-algebra $A$ equipped with a conditional expectation
$\phi:\ A\to D$,  there are separable unital $C^*$-subalgebras $\tilde A\subset A$ and $\tilde D\subset D\cap \tilde A$ such that 
$\phi_{|\tilde A}$ is a conditional expectation onto $\tilde D$.
Using this, we may assume all our Hilbert modules   countably generated
 and, as in \cite{RX}, we can use
Kasparov's absorption Theorem,  for which we refer to \cite[p. 60]{La}.
The latter says that for any countably generated Hilbert module $H$
there exists a (contractive and $D$-modular) factorization of the identity of $H$
through the standard (column) module
formed of sequences $x=(x_n)\in D^{\N_*}$
such that the series $\sum x_n^*x_n $ converges in norm in $D$,
equipped with the $D$-valued inner product
such that $\langle x,x \rangle= \sum x_n^*x_n $.
We denote this module by $C(D)$. Let $C_n(D)$ be the submodule
formed of all $x=(x_n)\in D^{\N_*}$ supported in $[1,n]$. Then
the union $\cup_n C_n(D)$  is norm dense in  $C(D)$.
It follows that there is a sequence
of contractive module mappings $F_n :\ H \to C_n(D)$
and  $G_n :\ C_n(D) \to H$ such that
  \begin{equation}\label{x27} G_nF_n(x)\to x \text{ for any } x\in H \text{
and }(G_nF_n)^{op}(y)\to y \text{ for any } y\in H^*.
\end{equation}
We will apply this to $H=L_2(\phi)$.
Recall that the algebraic free product $\cl A$ is such that
${\cl A} \xi $ is dense in $H$. We may assume that $\|F_n\|<1$
and $\|G_n\|<1$. Then by an elementary approximation argument
we may assume that, for each $n$, there are 
$\{ p^n_j\mid 1\le j\le n\} \subset \cl A$
and 
$ \{ q^n_j\mid 1\le j\le n\} \subset \cl A$ such that
$$\forall x\in H\ \forall  (c_j) \in C_n(D)\quad F_n(x)=(\phi( {p^n_j}^*x )) \text{ and } 
G_n((c_j) )= \sum\nolimits_j  q^n_j  c_j\xi .$$
Note that 
$$\|F_n\|= \| \sum\nolimits_j  \phi( {p^n_j}^*  p^n_j )\|^{1/2}
\text{ and }  \|G_n\|= \| \sum\nolimits_j \phi( q^n_j {q^n_j} ^*  )\|^{1/2}.
$$
By the pointwise convergence   $\phi_k\to \phi$, for any fixed $n$ we have
$$\lim\nolimits_{k,\U}  \| \sum\nolimits_j  \phi^{(k)}( {p^n_j}^*  p^n_j )\|^{1/2}<1
 \text{ and } \lim\nolimits_{k,\U}  \| \sum\nolimits_j  \phi^{(k)}(q^n_j {q^n_j}^* )\|^{1/2}<1 .$$
 Equivalently, if we define $F^{(k)}_n:\ H^{(k)}\to  C_n(D)$
 and  $G^{(k)}_n:\ C_n(D)\to  H^{(k)}$ by\\
  $F^{(k)}_n(x\xi^{(k)})=(\phi^{(k)}( {p^n_j}^*x ))$
 and $G^{(k)}_n((c_j))=\sum\nolimits_j q^n_j  c_j \xi^{(k)} $. Then
 for all $k$ large enough we have
 \begin{equation}\label{x16}    \| F^{(k)}_n \|_{\B(H^{(k)}, C_n(D))}<1
 \text{ and }  \| G^{(k)}_n \|_{\B(C_n(D) , H^{(k)}  ) }=\| G^{(k)op}_n \|_{\B({H^{(k)}}^*, {C_n(D)}^*)} <1 .\end{equation}
\begin{lem}\label{x17}
Fix $\vp>0$ and $N\ge 1$. Let $x(\alpha), y(\alpha)\in \cl A$  $(1\le \alpha \le N)$,
Then there is an  $n$ such that
\begin{equation}\label{x18}\forall  \alpha \le N\quad  \| x(\alpha) \xi  -G_nF_n x(\alpha) \xi\|_H<\vp
\text{ and }  \| \xi^{op} y(\alpha)  -(G_nF_n)^{op} \xi^{op} y(\alpha) \|_{H^*}<\vp \end{equation}
and moreover such that for all $k$ large enough we have
 $$ \| x(\alpha) \xi^{(k)}\|_{H^{(k)}} \le (1+\vp) \| x(\alpha) \xi\|_{H }\text{ , } 
\| \xi^{(k){op}} y(\alpha) \|_{{H^{(k)}}^*} \le (1+\vp) \|\xi^{op} y(\alpha)\|_{H^* } $$
$$  \| x(\alpha) \xi^{(k)} -G^{(k)}_nF^{(k)}_n x(\alpha) \xi^{(k)}  \|_{H^{(k)}}<\vp
\text{ and } 
 \| \xi^{(k){op}} y(\alpha) -(G^{(k)}_nF^{(k)}_n)^{op} \xi^{(k){op}}y(\alpha) \|_{{H^{(k)}}^*} <\vp.$$
\end{lem}
\begin{proof}   
For any $ x\in \cl A$, and any $n$ we have $$G_nF_n x \xi  = \sum_j q^n_j  \phi( {p^n_j }^*x) \text{  and  } 
G^{(k)}_nF^{(k)}_n x \xi^{(k)}  = \sum_j q^n_j  \phi^{(k)}( {p^n_j }^*x).$$
 Therefore
\begin{equation}\label{x19}\lim_\U \|  x \xi^{(k)} -G^{(k)}_nF^{(k)}_n x \xi^{(k)}\|_{H^{(k)}}= \| x  \xi  -G_nF_n x  \xi\|_H
\text{ and } \lim_\U \|  x \xi^{(k)} \|_{H^{(k)}}= \| x  \xi    \|_H
.\end{equation} 
By \eqref{x27}, we can choose an
 $n$ large enough such that \eqref{x18} holds. Then the other conditions
 can be achieved using \eqref{x19}. 
\end{proof}
\begin{proof}[Proof of Theorem \ref{t3a}] By the above observations,
it suffices to check that we still have \eqref{x5}.\\
We may assume $t_r(b_d)=\sum\nolimits_\alpha x(\alpha) \xi \otimes \xi^{op} y(\alpha)$, so that
$$( F_n\otimes G_n^{op})[ t_r(b_d)]=\sum\nolimits_\alpha  F_n(x(\alpha) \xi )\otimes  G_n^{op}(\xi^{op} y(\alpha))\in C_n(D)\otimes{C_n(D)}^* .$$
The latter defines a mapping in $\B(C_n(D))$  defined by
$$(d_j)\mapsto (\sum\nolimits_j a_{ij} d_j)_i$$
where
 \begin{equation}\label{x30}   a_{ij}=\sum\nolimits_\alpha  \phi ({p^n_i}^* x (\alpha))\phi(   y(\alpha) {q^n_j}).\end{equation}
By Lemma \ref{x17} there is an  $n$ such that
$$      \| t_r(b_d)\|_{\B(H)}  \le (1+\vp)  \| (G_nF_n)\otimes (G_nF_n)^{op}[ t_r(b_d)]\|_{\B(H)}
$$
and such that for all $k$ large enough (recall that, by \eqref{as55}, $v^{(k)} b_d$
is just a perturbation of $b_d$)
$$      \| t^{(k)}_r(v^{(k)} b_d)\|_{\B(H^{(k)})} \le (1+\vp)  \| (G^{(k)}_nF^{(k)}_n)\otimes (G^{(k)}_nF^{(k)}_n)^{op}[ t^{(k)}_r(v^{(k)} b_d)]\|_{\B(H^{(k)})}.
$$
Since $F_n,G_n$ are contractions 
\begin{equation}\label{x15}   
   \| t_r(b_d)\|_{\B(H)} \le (1+\vp)  \|  ( F_n\otimes G_n^{op})[ t_r(b_d)]\|_{\B(C_n(D))}
\le  (1+\vp) \| t_r(b_d)\|_{\B(H)}.
\end{equation}
and using \eqref{x16}
\begin{equation}\label{x21}      \| t^{(k)}_r(v^{(k)} b_d)\|_{\B(H^{(k)})} \le (1+\vp)  \|  ( F^{(k)}_n\otimes  G^{(k)op}_n)[ t^{(k)}_r(v^{(k)} b_d)]\|_{\B(C_n(D))}
\le  (1+\vp)  \| t^{(k)}_r(v^{(k)} b_d)\|_{\B(H^{(k)})}.
\end{equation}
But now assuming $t_r(b_d)=\sum_\alpha x(\alpha) \xi \otimes \xi^{op} y(\alpha)$, then 
$( F_n\otimes G_n^{op})[ t_r(b_d)]$ acts on $C_n(D)$
as the matrix $a=[a_{ij}]\in M_n(D)$ defined by \eqref{x30}.
Thus we have by \eqref{x15} 
$$ \| t_r(b_d)\|_{\B(H)} \le (1+\vp)  \| [a_{ij}]\|_{M_n(D)}
\le  (1+\vp) \| t_r(b_d)\|_{\B(H)}.$$
Let
$$a^{(k)}_{ij}= \sum\nolimits_\alpha  \phi^{(k)} ({p^n_i}^* [v^{(k)} x (\alpha)  ] )\phi^{(k)}     ( [v^{(k)} y(\alpha) ] {q^n_j} ).$$
Then by \eqref{x21}
$$ \| t^{(k)}_r(v^{(k)} b_d)\|_{\B(H^{(k)})} \le (1+\vp)  \| [a^{(k)}_{ij}]\|_{M_n(D)}
\le  (1+\vp) \| t^{(k)}_r(v^{(k)} b_d)\|_{\B(H^{(k)})}.$$
But since $\phi^{(k)} \to \phi$ pointwise on $\cl A$,
we have $\lim\nolimits_\U a^{(k)}_{ij}=a_{ij}$ and hence
passing to the limit in $k$ in the last two equivalences
we obtain 
$$ \| t_r(b_d)\|_{\B(H)} \le (1+\vp) \lim\nolimits_\U  \| t^{(k)}_r(v^{(k)} b_d)\|_{\B(H^{(k)})}\le (1+\vp)^2   \| t_r(b_d)\|_{\B(H)}.$$
Since $\vp>0$ is arbitrary, this establishes the first part of \eqref{x5}. \\
The second part can be checked by a similar argument: Using $F_n,G_n$
we are led to compare the norm of  $s_r(b_d)$
in $H_C \otimes_{hD} [A_1\oplus A_2] \otimes_{hD} (H^*)_R$ (as described in \cite[Prop. 5.1]{RX})
with the norm in $C_n(D)\otimes_{hD} [A_1\oplus A_2] \otimes_{hD} C_n(D)^*$, but the latter is isometric to $M_n(A_1\oplus A_2)$.
We skip the remaining details. \end{proof}
 
\bigskip
 
\bigskip
 
\n\textit{Acknowledgement.}  I am very grateful to Claus Koestler,
 \'Eric Ricard, Mikael de la Salle  and Paul Skoufranis for useful communications.

 \end{document}